\documentclass[11pt]{amsart}
\input{amssym.def}
\input{amssym}
\newtheorem{pro}{Proposition}[section]
\newtheorem{lem}[pro]{Lemma}

\newtheorem{exa}[pro]{Example}
\newtheorem{theo}[pro]{Theorem}
\newtheorem{defi}[pro]{Definition}
\newtheorem{cor}[pro]{Corollary}
\newtheorem{remk}[pro]{Remark}

\newcommand{\ep}{\varepsilon}
\newcommand{\al}{\alpha}

\newcommand{\vp}{\varphi}
\newcommand{\la}{\lambda}

\newcommand{\sun}{\odot}

\newcommand{\lra}{\longrightarrow}
\newcommand{\lmt}{\longmapsto}

\newcommand{\nrm}[1]{\mbox{ $ \displaystyle \left\| {#1} \right\| $} }
\newcommand{\nri}[1]{\mbox{ $ \nrm{ {#1} }_{\infty} $} }

\newcommand{\fk}[1]{ \left( {#1} \right) }

\newcommand{\bk}[1]{ \left\{ {#1} \right\} }
\newcommand{\btr}[1]{\mbox{ $ \left| {#1} \right| $ }}

\newcommand{\ce}{{\bf\Bbb C}}

\newcommand{\topowsOT}{{w^*}}

\newcommand{\re}{{\bf\Bbb R}}
\newcommand{\rep}{{\bf\Bbb R^+}}
\newcommand{\za}{{\bf\Bbb N}}

\newcommand{\jz}{{\bf\Bbb J}}

\newcommand{\topoks}{{w^*}}

\newcommand{\scS}{{\mathcal{S}}}

\newcommand{\scST}{{\mathcal{T}}}
\newcommand{\scSTO}{{\overline{\scST}^{\topoks}}}
\newcommand{\scSTS}{{\scST^{\sun}}}
\newcommand{\scSTOS}{\overline{\scSTS}^{\topoks}}

\newcommand{\scTO}{{\overline{\mbox{co}}^{\topoks}\scST}}
\newcommand{\scTOS}{{\overline{\mbox{co}}^{\topoks}\scSTS}}

\newcommand{\scATO}{{\overline{\mbox{ac}}^{\topoks}\scST}}

\newcommand{\LT}{{L_{\scST}}}

\newcommand{\scJ}{{\mathcal{J}}}

\newcommand{\Xsr}{X^{\sun}_{rev}}

\newcommand{\semig}[2]{\bk{{#1}(t)}_{t\in {#2} }}

\newcommand{\seq}[2]{\mbox{$ \bk{ #1_{#2} }_{{#2} \in \za} $} }
\newcommand{\net}[3]{\mbox{$ \bk{ {#1}_{#2} }_{{#2} \in {#3}} $} }

\newcommand{\supp}[1]{ \mbox{ $ {\rm supp} \left\{ {#1} \right\} $ } }

\newcommand{\ilm}[1]{  \lim_{ {#1} \to \infty}  }
\newcommand{\netlim}[2]{  \lim_{ {#1}\in {#2}}  }

\newcommand{\Funk}[5]{ \begin{array}{ccccc}
                       {#1} & : & {#2} & \lra & {#3} \\
                            &   & {#4} & \lmt & \displaystyle{#5} 
                       \end{array}                       }
\newcommand{\funk}[3]{ \begin{array}{ccccc}
                       {#1} & : & {#2} & \lra & {#3}
                       \end{array}                       }

\begin{document}
\title{Compactification of bounded semigroup representations}
\author{Josef Kreulich, \\ Universit\"at Duisburg-Essen}
%\institute{Josef.Kreulich@gmail.com}

\maketitle
\begin{abstract}
This study uses methods to identify compactifications of semigroups $S\subset L(X)$ that reside in the space $L(X).$ These methods generalize, in some sense, the deLeeuw-Glicksberg theory to a greater class of vectors. This provides an abstract approach to several notions of almost periodicity, mainly involving right semitopological semigroups \cite{RuppertLNM} and adjoint theory. Moreover, the given setting is refined to the case of bounded $C_0-$semigroups.
\end{abstract}

\section{Introduction}
The main idea of this work is the result that $L(Y,X^*)$ is a dual space and therefore has a $w^*-$topology. Consequently, for subspaces $Y\subset X^*,$ the operator space $L(Y)\subset L(Y,X^*)$ has a $w^*-$topology. This result applies to several cases, such as sun-dual semigroups $$X^{\sun}:=\bk{x^*\in X^*:\nrm{\cdot}-\lim_{t\to 0}T^*(t)x^*=x^*}\subset X^*$$ and bounded uniformly continuous functions. The cases support results for several notions of almost periodicity and therefore ergodic results as well. The advantage of the compactification residing in the same operator space is that the idempotent becomes a projection on the given space, which serves to split the space into so-called reversible and flight vectors. In this scope, we consider the weak-star topology instead of the weak topology and address, by necessity, the well-known fact that bounded linear operators between dual spaces are
%Editor2: On the line below, please ensure that the intended meaning has been maintained in this edit.
$weak-weak$-continuous but
not necessarily $weak^*-weak^*$-continuous. This approach provides answers based on locally convex theory to enlarge the boundaries within which results on almost periodicity are true. A brief overview of why some results apply is given. 

The used approach extends several splittings, and therefore asymptotic considerations to more general spaces and therefore Cauchy problems considered in \cite{RuessSemi},\cite{RuessInt}, \cite{RuessErgdc}.

\insert\footins{\footnotesize The author thanks Professor Ruess for his suggestions and advice. \\
Josef.Kreulich@gmail.com}

\section{Preliminaries and Notation}

The main literature on functional analysis is \cite{RudinFA} on $C_0-\mbox{semigroups}$; please refer to \cite{HillePhillips}, \cite{Pazy}, and \cite{Nagel}, and for special results on sun-dual semigroups, to \cite{HillePhillips} and \cite{NeervenLNM}. To obtain the main definitions and results of deLeeuw-Glicksberg theory \cite{DeGli1} and \cite{DeGli2}, \cite[pp. 103ff, § 2.4]{Krengel} is sufficient. The results of harmonic analysis are taken from \cite{HewittRoss}, the general topology results on nets from \cite{Kelley}, and the locally convex topological space results from \cite{Jarchow} and \cite{Koethe2}. For the special results on right semitopological semigroups, refer to \cite{RuppertLNM}.

For the notions of locally convex topology, recall that
$\sigma(X,Y)$ is the weakest topology such that the space of continuous linear functionals
$(X,\sigma(X,Y))^*=Y$ (i.e., the $w$-topology =$\sigma(X,X^*)$, and the $w^*$-topology = $\sigma(X^*,X)$).

\section{Dual Representations and Left-Semitopological Semigroups}
In this section we give the general setting, it fit into general theory, and two main results on the compactness of subgroups, and a condition on $\semig{T}{S}$ which implies the abelian structure for the compactification.
We start this section with a remark that points to the main topological ideas of this study. It is the interaction of two dual spaces, the dual $X^*$ itself, and $L(X,X^*).$ This point of view gives lead to structures and ergodic results for the semigroup.

\begin{remk} \cite[ Prop. 3, p.330]{Jarchow} Let $X,Y$ be Banach spaces, and let $L(X,Y)$ be the Banach space of bounded and linear operators.
As $L(X,Y^*)=(X\otimes_{\pi}Y)^*,$ a bounded net $\net{T}{\la}{\Lambda}\subset L(X,Y^*)$ is convergent to $T\in L(X,Y^*)$ in the $w^*-$topology if
$$
\netlim{\la}{\Lambda}<T_{\la}x,y>=<Tx,y> \mbox{ pointwise on } x\in X, \ y \in Y.
$$
For this study, it is only of interest that $B_{L(X,Y^*)}$ is compact with respect to the predefined topology, which can be verified by an application of Tychonov's theorem. This implies that $L(X,Y^*)$ is a dual space, and on $B_{L(X,Y^*)}$, the $w^*-$topology and the predefined topology coincide; compare \cite{Kaiser}.

\end{remk}

As we discuss semigroups with a certain continuity property, we recall

\begin{defi}\cite[p. 8, Definition 1.2, 1.3]{RuppertLNM}
\begin{enumerate}
\item Let $A, B $ and $C$ be topological Hausdorff spaces. A map
$$\Funk{\pi}{A\times B}{C}{(x,y)}{\pi(x,y)} $$
is called right (left) continuous if it is continuous in the right (left) variable; that is, for every fixed $x_0\in X$ ($y_0\in Y$), the map $\bk{y\mapsto \pi(x_0,y)}$ $(\bk{x\mapsto\pi(x,y_0)})$ is continuous.
\item Let $\scS$ be a nonvoid topological space that is provided with an associative multiplication
$$
\Funk{ \mu}{\scS \times \scS}{\scS}{(x,y)}{\mu(x,y)=xy.}
$$
Then, the pair $(\scS,\mu)$ is called a right (left) semitopological semigroup if $\mu$ is right (left) continuous.
\end{enumerate}
\end{defi}

\begin{defi} \cite[p. 103]{Krengel}
A semigroup is called (semi)-topological is a semigroup with unit which is a Hausdorff space in which the multiplication is (separately) jointly continuous.
\end{defi}

Except for the section Ideal Theory, where $S$ is only right(left)-topological and not necessarily Abelian,  we assume $S$ to be an Abelian semitopological semigroup. 

The next definition can be derived from the idea of considering the semigroup as a restriction of a dual operator.
\begin{defi} \label{defi-dual-representation}
For an Abelian topological semigroup $S$, a Banach space $X,$ and $Y\subset X^*$, a set $\semig{T}{S}\subset L(Y)$ is called a dual semigroup representation if
\begin{enumerate}
\item $T(s+t)=T(s)T(t)$;
\item for given $x\in X$ and $y\in Y$,
$$
\Funk{T_y}{S}{\ce}{s}{<x,T(s)y>}
$$
is continuous;
\item $\overline{\bk{T(t)x:t\in S}}^{w^*}\subset Y$;
\item for all $s\in S$ and $y\in Y,$
$$
\Funk{T(s)}{(\overline{ac\bk{T(t)y:t\in S}}^{w^*},w^*)}{(\overline{ac\bk{T(t)y:t\in S}}^{w^*},w^*)}{x}{T(s)x}
$$
is continuous.
\end{enumerate}
\end{defi}

\begin{remk}
Note, that in general the operator need not to be $w^*-w^*-$continuous, and therefore need not to be a restriction of dual operator. The definition mainly weights topological attributes of the orbit against the continuity properties of the operator. For example, choose, $S=\za,$ $T\in L(X^*)$ with $T$ compact, or weakly compact, $\nrm{T}\le 1,$ and $T(n)=T^n.$ Then 
$\overline{ac\bk{T(t)y:t\in S}}^{w^*}=\overline{ac\bk{T(t)y:t\in S}}^{\nrm{\cdot}},$ 
or $=\overline{ac\bk{T(t)y:t\in S}}^w$ and the norm, or the weak topology coincide with the $w^*-$topology.
\end{remk}

\begin{defi}
\begin{enumerate}
\item Let $\semig{T}{S}\subset L(Y)$ be a dual representation of a semigroup $S.$
$y\in Y$ is called reversible if for every net $\net{s}{\al}{A}\subset S$, there exists a net
$\net{t}{\gamma}{\Gamma}\subset S$ such that
$w^*-\netlim{\gamma}{\Gamma}w^*-\netlim{\al}{A}T(t_{\gamma})T(s_{\al})y=y.$ Let $Y_{rev}$ be the set of reversible vectors.
\item $y\in Y$ is called a flight vector if for a net $\net{s}{\al}{A}\subset S, $ we have $w^*-\netlim{\al}{A}T(s_{\al})y=0.$ Let $Y_0$ be the set of flight vectors.
\end{enumerate}
\end{defi}

With the above definition,
%Editor2: On the line below, please ensure that the intended meaning has been maintained in this edit.
the following becomes
straightforward.
\begin{pro} \label{rev-equals-flight-equals-zero}
If $x\in Y$ is a flight vector and reversible, then $x=0.$
\end{pro}

\begin{pro} \label{first-results}
Let $\semig{T}{S}$ be a dual semigroup representation.
\begin{enumerate}
\item $(\semig{T}{S},w^*)$ is a left semitopological semigroup.
\item $\semig{T}{S}$ has a compactification $\scST=\overline{\semig{T}{S}}^{w^*} \subset L(Y)$, which is a left semi-topological semigroup.
\item $\overline{co}^{w^*}\semig{T}{S}\subset L(Y)$ is a compact left semitopological semigroup.
\item  $\overline{ac}^{w^*}\semig{T}{S}\subset L(Y)$ is a compact left semitopological semigroup, and $T(s)$  commutes with every operator $U\in \overline{ac}^{w^*}\semig{T}{S}.$
\item $Y=Y_{a}\oplus Y_0$, where $Y_a\subset Y_{rev}$ and $Y_0\subset Y_{fl}.$
\end{enumerate}
\end{pro}
\begin{proof}
Let $\net{s}{\al}{A}$ be a $w^*-$ convergent net with limit $s$, $y\in Y.$ Then, for a given $V\in L(Y)$, we have
$$
\netlim{\al}{A}<x,T(s_{\al})Vy>=<x,T(s)Vy>.
$$

%Editor2: On the line below, please ensure that the intended meaning has been maintained in this edit here and elsewhere throughout the manuscript.
For nets $\net{t}{\gamma}{\Gamma}$ and $\net{s}{\al}{A}\subset S,$ let $R=w^*- \lim_{\gamma}T(t_{\gamma})$ and $S=w^*-\lim_{\alpha}T(s_{\alpha}).$
From Def. \ref{defi-dual-representation}(3), we have $R,S\in L(Y),$ and from Def. \ref{defi-dual-representation}(4), we have
$$
w^*- \netlim{\gamma}{\Gamma} w^*-\netlim{\alpha}{A}  T(t_{\gamma})T(s_{\alpha})= w^*- \netlim{\gamma}{\Gamma} T(t_{\gamma} )S=RS,
$$
which proves $RS\in \overline{\bk{T(t)}_{t\ge0}}^{w^*}.$ The left continuity is straightforward. For the convex or absolute convex hull, the proof is quite similar.

Let $s\in S$ and $R=w^*- \lim_{\gamma}T(t_{\gamma}).$ Then, due to \ref{defi-dual-representation}(4),
$$T(s)R=w^*- \netlim{\gamma}{\Gamma}T(s)T(t_{\gamma})=w^*- \netlim{\gamma}{\Gamma}T(t_{\gamma})T(s)=RT(s).$$

Applying Thm. \ref{right_topo_semigroups}, we obtain a minimal idempotent $P$; hence, $Y= P Y\oplus (I-P)Y.$
Let $\net{s}{\al}{A}\subset $ with $U=w^*-\lim_{\alpha}T(s_{\alpha}).$ Applying Thm. \ref{right_topo_semigroups}(4), we find $PVP\in\scST$ such that $PVPUPy=Py$; hence, $Py\in Y_{rev}.$ The identity $P(I-P)y=0$ proves $(I-P)Y\subset Y_{fl}.$

\end{proof}

\begin{exa}
Let $\semig{T}{\rep}\subset L(X)$ be a $C_-$semigroup, where $\semig{T^{\sun}}{\rep}$ is a dual representation on $X^{\sun}.$
\end{exa}
\begin{proof}
Clearly, $\semig{T^{\sun}}{\rep}$ are $w^*-w^*$-continuous operators.
Let $S\in \scST$; then, for $x\in X^{\sun},$ we have $ T^{\sun}(t)Sx=ST^{\sun}(t)x\to Sx$ as $t\to 0.$ Hence
$$\overline{O(x)}^{w^*}=\bk{Sx:S\in\scST}\subset X^{\sun},$$
which concludes the proof.
\end{proof}

\begin{theo} \label{PTP-is-compact-group}
Let $P$ be a minimal idempotent.
\begin{enumerate}
\item If for all $y\in Y,$
$$
\Funk{P}{(\overline{\bk{T(t)y:t\in S}}^{w^*},w^*)}{(\overline{\bk{T(t)y:t\in S}}^{w^*},w^*)}{x}{Px}
$$
is continuous, then $P\scST P$ is a compact group.
\item If $P=V^*_{|Y }$ with $V\in L(X),$ then $P\scST P$ is a compact group.
\item If a minimal idempotent $P$ commutes with every operator of the compactification, then $P\scST P$ is a compact group.
\end{enumerate}
\end{theo}
\begin{proof}
To verify the compactness of group $P \scST P,$ let $\net{T}{\al}{A}$ be a net. Due to the compactness of $\scST $, we have the cluster point $T$. Without loss of generality, $T$ is the limit. Then, $\bk{T_{\al}Px}\subset \overline{ac\bk{T(t)Px:t\in S}}^{w^*},$ $w^*-\netlim{\al}{A}T_{\al}Px=TP.$ Consequently,
$$\netlim{\al}{A} <x,PT_{\al}Px>=<x,PTPx>.$$
which proves the claim. Dual operators are $w^*-w^*-$continuous. The case in which $P$ commutes is straightforward.
\end{proof}

\begin{cor}
If $\scST$ is Abelian, then $P\scST P$ is a compact Abelian group.
\end{cor}
\begin{proof}
Use the fact that the minimal idempotent $P \in \scST.$
\end{proof}

Next, we recall some definitions of certain classes of vectors and functions.
\begin{defi} Let $\semig{T}{S}$ be a dual representation on $Y\subset X^*.$
\begin{enumerate}
\item The orbit of a vector $y\in Y$ is given by
$$
O(x):=\bk{T(t)x:t\in S}.
$$
\item A vector $x\in Y$ is called Eberlein-weak almost periodic (E.-wap) if
$
O(x)$ is weakly relatively compact in $Y.$
\item A vector $x\in Y$ is called almost periodic if $S= \re$ and
$O(x)$ is relatively compact in $Y.$
\item We define for a Banach space $X$ that a function $f\in C_b(S,X)$ is called E.-wap if the orbit with respect to the translation semigroup is rel.~ $\sigma(C_b(\re,X),C_b(\re,X)^*)-$compact.
\begin{eqnarray*}
W(S,X)&:=&\bk{f\in C_b(S,X): f \mbox{ is E.-wap}}, \\
W_0(S,X)&:=& \bk{f\in W(S,X): f_{t_n} \to 0 \mbox{ weakly for some } \seq{t}{n}\subset \jz },\\
AP(\re,X)&:=&\bk{f\in C_b(\re,X): f \mbox{ is almost periodic}}.
\end{eqnarray*}
In the scalar-valued case, $X$ is left.
\item Let $M \subset Y^*$ be a closed and separating space and $S=\re$; a vector $y\in Y$ is called $M-$weakly almost periodic if $\bk{t\mapsto x^*T(t)x} \in AP(\re)$ for all $x^*\in M.$

\end{enumerate}
\end{defi}

\begin{theo} \label{T-abelian-condition}
Let $S$ be Abelian and $\semig{T}{S}$ be a dual representation on $Y$. If $y\in Y,$ we have $UVy=VUy$ for all $U,V\in \scST$ if and only if
$$
\Funk{i_x}{S}{W(S)}{s}{<x,T(s)y>}
$$
for a set of separating vectors $x\in X.$
\end{theo}
\begin{proof}
For nets $\net{t}{\gamma}{\Gamma},\net{s}{\al}{A}\subset S, let $
$R=w^*- \lim_{\gamma}T(t_{\gamma})$ and $S=w^*-\lim_{\alpha}T(s_{\alpha}).$
The identity
$$
<x,RSy>=\netlim{\alpha}{A}\netlim{\alpha}{A}<x, T(t_{\gamma}+s_{\alpha})y>
=\netlim{\alpha}{A}\netlim{\alpha}{A}<x, T(t_{\gamma}+s_{\alpha})y>=<x,SRy>
$$
proves the claim in both directions. If $<x,T(s)y>\in W(S),$ the double limit criteria hold, and the operators commute; if they commute, the double limit criteria \cite{Groth} imply $<x,T(s)y>\in W(S).$

Now, let $U,V\in\overline{ac}^{w^*}\semig{T}{S}.$ Then, we find nets
$\bk{t_i^{\gamma}}_{i\in\za,\gamma\in\Gamma},$ $\bk{s_i^{\lambda}}_{i\in\za,\lambda\in\Lambda}\subset\rep$
and $\bk{\alpha_i^{\gamma}}_{i\in\za,\gamma\in\Gamma},$ $\bk{\beta_i^{\lambda}}_{i\in\za,\lambda\in\Lambda}\subset \re$,
with
$\sum_{i=1}^{n_{\gamma}} \btr{\alpha_i^{\gamma}}\le1$
and
$\sum_{i=1}^{m_{\lambda}}\btr{\beta_i^{\lambda}}\le 1,$ such that
$$
U_{\gamma}:=\sum_{i=1}^{n_{\gamma}}\alpha_i^{\gamma}T(t_i^{\gamma})\mbox{ with } \netlim{\gamma}{\Gamma}U_{\gamma}=U \mbox{ and }
V_{\lambda}=\sum_{i=1}^{m_{\lambda}}\beta_i^{\lambda}T(s_i^{\lambda}), \mbox{ with } \netlim{\lambda}{\Lambda}V_{\lambda}=V.
$$
Now, define $g(s):=<x,T(s)x^*>,$, which is assumed to be E.-wap, and the bounded linear functional $\delta_t(g):=g(t)$; then, the duality reads as
\begin{eqnarray*}
\lefteqn{<x,VUx^{*}>=\netlim{\lambda}{\Lambda}<x,V_{\la}U^{*}x^{*}>} \\
&=&\netlim{\lambda}{\Lambda}\netlim{\gamma}{\Gamma}
\sum_{i=1}^{n_{\gamma}}\sum_{j=1}^{m_{\lambda}}\alpha_i^{\gamma}\beta_j^{\lambda}<x,T(t_i^{\gamma}+s_{j}^{\lambda})x^{*}>,\\
&=&\netlim{\lambda}{\Lambda}\netlim{\gamma}{\Gamma}
\sum_{i=1}^{n_{\gamma}}\sum_{j=1}^{m_{\lambda}}\alpha_i^{\gamma}\beta_j^{\lambda}g(t_i^{\gamma}+s_{j}^{\lambda}), \\
&=&\netlim{\lambda}{\Lambda}\netlim{\gamma}{\Gamma}
<\sum_{i=1}^{n_{\gamma}}\alpha_i^{\gamma}g(\cdot+t_i^{\gamma}),\sum_{j=1}^{m_{\lambda}}\beta_j^{\lambda}\delta_{s_j^{\lambda}}>_{(C_b(S),C_b(S)^*)}\\
&=&\netlim{\gamma}{\Gamma}\netlim{\lambda}{\Lambda}
<\sum_{i=1}^{n_{\gamma}}\alpha_i^{\gamma}g(\cdot+t_i^{\gamma}),\sum_{j=1}^{m_{\lambda}}\beta_j^{\lambda}\delta_{s_j^{\lambda}}>_{(C_b(S),C_b(S)^*)}.
\end{eqnarray*}
As $O(g)$ is weakly relatively compact in $C_b(S),$
its closed absolutely convex hull is weakly compact. Furthermore, because $\nrm{\delta_t}\le 1,$ the absolute convex combination is bounded. Hence, we have separated the limits and determined that the interchanged limits coincide.
\end{proof}

Next, we show how some results of \cite{DeGli1} embed into the given context.
\begin{pro} \label{embed-DeGli-theory}
Let $j_X:X\to X^{**}$ denote the canonical embedding, $S$ be an Abelian semigroup, and $X$ be a Banach space $\semig{T}{S}\subset L(X)$ such that
\begin{enumerate}
\item $T(s+t)=T(s)T(t)$;
\item for given $x\in X,$ $x^* \in X^*$
$$
\Funk{T_x}{S}{\ce}{s}{<x^*,T(s)x>}
$$
is continuous.
\item $\bk{T(t)x:t\in S}$ is weakly relatively compact.
\end{enumerate}
Then, $\semig{T^{**}_{|j_XX}}{S}$ is a dual representation on $j_XX,$ and for all $x^*\in X^*,x\in X,$ we have
$$
\bk{s\mapsto <x^*,T^{**}(s)j_Xx>}\in W(S).
$$
\end{pro}
\begin{proof}
Because $T^{**}(t)j_Xx=j_XT(t)x,$ and due to the weak compactness of the orbit, $\overline{O(j_Xx)}^{w^*}=\overline{O(j_Xx)}^{w}\subset j_XX.$ The rest is straightforward.
\end{proof}

\begin{pro} \label{weak-almost-periodicity-implies-X-a}
Let $\jz\in\bk{\re,\rep}$ and $x^{\sun}\in X^{\sun}.$
\begin{enumerate}
\item Let $<x,T^{\sun}(\cdot)x^{\sun}>\in W(\jz)$; then, for the splitting $x^{\sun}=x_a^{\sun}+x_0^{\sun},$ we have $<x,T^{\sun}(\cdot)x^{\sun}_a>\in AP(\re)_{|\jz}$ and $<x,T^{\sun}(\cdot)x^{\sun}_0>\in W_0(\jz).$
\item If $<x,T^{\sun}(\cdot)x^{\sun}>\in AP(\re)_{|\jz}$ for a separating set of vectors $x\in X,$ then $x^{\sun}\in X_a^{\sun}.$
\item If $<x,T^{\sun}(\cdot)x^{\sun}>\in W_0(\jz)$ for a separating set of vectors $x\in X,$ then $x^{\sun}\in X_0^{\sun}.$
\end{enumerate}
\end{pro}
\begin{proof}
As $<x,T^{\sun}(t)x^{\sun}>=<x,T^{\sun}(t)x_a^{\sun}>+<x,T^{\sun}(t)x_0^{\sun}>,$
we learn from the definition of the vector-valued splitting that $<x,T^{\sun}(t)x_0^{\sun}>$ is a flight vector, and $<x,T^{\sun}(t)x_a^{\sun}>$ is reversible with respect to the scalar translation group on $W(\jz).$ The splitting of \cite{DeGli1} gives $W(\jz)_{rev}=AP(\re)_{|\jz},$ and $W(\jz)_{fl}=W_0(\jz).$

Using almost periodic functions that are reversible, we have
$$<x,T^{\sun}(t)x^{\sun}>-<x,T^{\sun}(t)x_a^{\sun}>=<x,T^{\sun}(t)x_0^{\sun}>, \mbox{ for all } x\in X, t\in\re.
$$
By Pro 8.17, we find that $<x,T^{\sun}(t)x^{\sun}>-<x,T^{\sun}(t)x_a^{\sun}>=0$ for a separating set of vectors $x\in X$ and all $ t\in\re,$ and therefore, $x^{\sun}=x_a^{\sun}.$ For the last claim, use
$<x,T^{\sun}(t)x^{\sun}>-<x,T^{\sun}(t)x_0^{\sun}>=<x,T^{\sun}(t)x_a^{\sun}>$ and similar arguments.
\end{proof}

As our main interest reduces to bounded sets in operator spaces such as $L(X,Y^*)$, for a given locally convex topology on a Banach space $X,$ we define the convergence of bounded nets.
\begin{defi}
Let $\tau$ be a locally convex topology on $X^{*}.$
\begin{enumerate}
\item $\tau$ is call representation compatible if for all $y\in Y$, the identity map
$$
\funk{id}{(\overline{ac\bk{T(t)y:t\in S}}^{w^*},w^*)}{(\overline{ac\bk{T(t)y:t\in S}}^{\tau},\tau)}
$$
is a homeomorphism.
\item Let $Y\subset X^{*}$; for the net $$\net{T}{\al}{A}: Y\to Y,$$
%Editor: Please ensure that the intended meaning has been maintained in this edit.
we call the net $\net{T}{\al}{A}$ $\tau-OT$ convergent if there exists a $T\in L(Y)$ such that
$$\tau-\netlim{\al}{A}T_{\la}x^{*}=Tx^{*} \mbox{ for all } x^{*}\in Y.
$$
\end{enumerate}
\end{defi}

\begin{pro} \label{acT-compact}
Let $\semig{T}{S}$ be a dual representation and $\tau$ be a representation-compatible topology. Then,
$$
\Funk{\iota}{\fk{\overline{ac}^{w^*}\fk{\semig{T}{S}},w^*}}{\fk{\overline{ac}^{\tau-OT}\fk{\semig{T}{S}},\tau_Y-OT}}{T}{T}
$$
is a homeomorphism, and $(\overline{ac}^{\tau-OT}\fk{\semig{T}{S}},\tau-OT)$ is a compact right semitopological semigroup.
\end{pro}
\begin{proof}
By definition, if $T\in\overline{ac}^{w^*}\fk{\semig{T}{S}},$ then we have $Tx \in \overline{ac\bk{T(t)x:t\in S}}^{w^*},$ which concludes the proof by an application of Tychonov.
\end{proof}

\section{Harmonic Analysis of Dual Representations}
The next results provide an abstract approach to these types of periodicity. Using the previous condition, which proves the Abelian structure, we apply the existing Haar measure, Pettis measurability criteria, and Mackey topology to obtain results in the style of \cite{DeGli1}. Additionally, two theorems give an answer, in some sense, regarding the distance between the reversible part and almost periodicity. This approach lead to two main theorems.

To identify almost periodicity in \cite{DeGli1}, the class of vectors below is defined.

\begin{defi}
For $\semig{T}{S}$ a dual representation, a vector $y\in Y$ is an eigenvector with a unimodular eigenvalue if for a map
$\funk{\la}{\scST}{\ce}$ with $\btr{\la(T)}=1$, we have $Ty=\la(T)y$ for all $T\in \scST$. We define
$$Y_{uds}:=\overline{span}\bk{y\in Y: y \mbox{ is an eigenvector with a unimodular eigenvalue }}.
$$
\end{defi}

\begin{pro} \label{start-harmonic-analysis}
Let $Y\subset X^*$, $\semig{T}{S}$ be a dual representation, $P$ be a minimal idempotent, $\tau$ be a representation-compatible topology, and
$$
\bk{S\ni t \mapsto <T(t)y,x>}\in W(S) \mbox{ for all } y\in Y, x\in X.
$$
Then, the minimal idempotent $P$ is unique, and $G:=P\scST P$ is a compact Abelian group. Furthermore, let $\Gamma $ be the character group og $G$, $\gamma\in \Gamma,$ and let $\rho$ denote the normalized Haar measure on $G.$ Then,
$$
S_{\gamma}:=\int_G\overline{\gamma}(S)Sd\rho(S)
$$
exists in the sense of \cite[Def 3.26, p. 74]{RudinFA} in $(L(Y),\tau-OT)$, and
$
S_{\gamma}\in \overline{ac}^{\tau-OT}\fk{\semig{T}{S}}.
$
For a given $V\in \overline{ac}^{\tau-OT}\fk{\semig{T}{S}},$ we have, for all $x\in Y$,
$$
S_{\gamma}Vx=\int_G\overline{\gamma}(S)S Vx d\rho(S)
$$
in $(Y,\tau).$
\end{pro}

\begin{proof}
Note that for two minimal idempotents, by definition, $P_1=P_1P_2=P_2P_1=P_2.$
From Thm \ref{PTP-is-compact-group}, we find that $G$ is semitopological, and from an abstract harmonic analysis \cite{Ellis_cont}, we recall that any compact semitopological group is a topological group. Hence, we find the normalized Haar measure $\rho$ on $G,$ \cite[Thm 5.14, p. 123]{RudinFA}.

To prove the existence of the integral, we apply Theorem \cite[Thm. 3.27, pp. 74-75]{RudinFA} with respect to topology $\tau-OT.$
$$
\Funk{f}{(G,\tau-OT)}{L(Y),\tau-OT)}{S}{\overline{\gamma(S)}S}
$$
is continuous, and $f(G)\subset \overline{ac}^{\tau-OT}\fk{\semig{T}{S}},$ which is compact by Prop. \ref{acT-compact}.

Consequently, the integral exists and is an element of $\overline{ac}^{\tau-OT}\fk{\semig{T}{S}}.$ For the additional proof, note that, for $x^{\sun}\in Y,$

$$
\Funk{\delta_{x}}{(L(Y),\tau-OT)}{(Y,\tau)}{S}{Sx}
$$
is continuous, and for $V\in L(Y)$,
$$
S_{\gamma}Vx=\delta_{Vx}(S_{\gamma})=\delta_{Vx}\fk{\int_G\overline{\gamma}(S)Sd\rho(S)}=\int_G\overline{\gamma}(S)SVxd\rho(S);
$$
the claim becomes a consequence of \cite[p.85, Exercise 24]{RudinFA}.

\end{proof}

\begin{theo} \label{Y_uds=Y_a}
Let $\semig{T}{S}$ be a dual representation, $\tau$ be a representation-compatible topology, and
$$
\bk{S\ni t \mapsto <T(t)y,x>}\in W(S) \mbox{ for all } y\in Y, x\in X.
$$

Then,
$$
\overline{Y_a}^{\tau}\subset \overline{Y_{uds}}^{\tau}.
$$
\end{theo}
\begin{proof}[Proof of Theorem \ref{Y_uds=Y_a}]
\underline{Part 1:} In this part, we prove that unimodular eigenvectors are found as an image of the generalized Fourier transforms for a given $\gamma\in \Gamma.$
Similar to Prop. \ref{start-harmonic-analysis}, for the minimal idempotents $P,$  we find that $\rho$ is the normalized Haar measure on the Abelian compact topological group
$G=P \scST P.$

Furthermore, if $\Gamma$ denotes the character group, then for $\gamma\in\Gamma$, we can define
$$
S_{\gamma}:=\int_G\overline{\gamma}(S)Sd\rho(S) \in\overline{ac}^{\tau-OT}\fk{\semig{T}{S}} \subset L(Y).
$$
Consequently, for $y\in Y$, we have $S_{\gamma}y\in Y,$ and because $S_{\gamma}\in \overline{ac}^{\tau-OT}\fk{\semig{T}{S}}$ and $\overline{ac}^{\tau-OT}\fk{\semig{T}{S}}$ is Abelian by Theorem \ref{T-abelian-condition}, we find that $S_{\gamma}$ commutes with the operators in $G \subset \overline{ac}^{\tau-OT}\fk{\semig{T}{S}}.$ Using Prop. \ref{start-harmonic-analysis} for $R\in G, y\in Y$, we find that
\begin{eqnarray*}
RS_{\gamma}y&=&S_{\gamma}Ry=\delta_{Rx}(S_{\gamma})
=\delta_{Ry}\fk{\int_G \overline{\gamma}(S)Sd\rho(S)} 
=\int_G \overline{\gamma}(S)SRyd\rho(S) \\
&=&\int_G \overline{\gamma}(S)RSyd\rho(S) 
=\gamma(R)\int_G \overline{\gamma}(RS)RSyd\rho(S) \\
&=&\gamma(R)\int_G \overline{\gamma}(S)Syd\rho(S) 
\mbox {  apply \cite[Thm 5.14 (1),(2), p. 123]{RudinFA}}\\
&=&\gamma(R)S_{\gamma}y.
\end{eqnarray*}
Similarly, using the fact that $G$ is Abelian, we obtain
\begin{equation}
RS_{\gamma}=\gamma(R)S_{\gamma}=S_{\gamma}R.
\end{equation}
Because $P$ is the unit of $G,$ we have $\gamma(P)=1,$ and by the previous observation,
$  PS_{\gamma}=S_{\gamma}.$ Hence, for $T\in \scST,$ we find $PT\in G$ and
\begin{equation} \label{M-is-T-invariant}
TS_{\gamma}=TPS_{\gamma}=\gamma(TP)S_{\gamma}=\gamma(T)S_{\gamma}.
\end{equation}
This means that $S_{\gamma}Y$ consists of eigenvectors with unimodular eigenvalues $\la(T)=\gamma(T).$

\underline{Part 2:} Let $\Gamma$ be the character group of $G.$

We prove that $Y_{a}$ cannot be separated from
\begin{equation} \label{Definition-of-space-M}
M=\overline{span}\bk{S_{\gamma}x:y \in Y, \ \gamma\in \Gamma}
\end{equation}
with a $\tau-$continuous functional $\phi$, and we apply Proposition \ref{tau-separated}.

Because $M\subset X_{uds}\subset X_{a},$ we assume that there is a $y\in Y_{a} \backslash M.$ By the assumption, we find a $\tau$-continuous $\phi$ such that for $\phi(Py)=\phi(y) \not= 0$, $\phi_{|M}=0.$ Because for $x\in Y,$ $\Lambda(T):=\phi(Tx)$ is $\tau-OT$-continuous, we obtain

\begin{equation} \label{G_integrals_eq_null}
0=\phi(S_{\gamma}z)=\int_G\overline{\gamma}(S)\phi(Sz)d\rho(S)
\end{equation}
for all $\gamma\in \Gamma$ and $z\in Y.$

Because the characters form an orthonormal basis in $L_2(G,\rho)$---see \cite[p. 944]{DS}---we have
$$\bk{G\ni S\mapsto \phi(Sy)} =0  \ a.e.$$

Because $G$ carries the topology $\tau-OT,$ for $\phi$ $\tau$-continuous and $ z\in Y$, the functions

$$
\Funk{g}{(G,\tau-OT)}{\ce}{S}{\phi(Sz)}
$$

are continuous. Consequently,$ \bk{G\ni S\mapsto \phi(Sy)}$ is also zero, and we find a contradiction to $\phi(Py)\not=0,$ which completes the proof.
\end{proof}

The above result suggests the almost periodicity of the reversible part of an E.-wap solution.
Letting $\mu(\tau)$ be the Mackey topology that comes with $\tau,$ we have, by Mazur's Theorem, the following.
%Editor: Please ensure that the intended meaning has been maintained in this edit.

\begin{cor} \label{semigroup-equicontinuous}
Let $\semig{T}{S}$ be a dual representation, $\tau$ be a representation-compatible topology, and
$$
\bk{S\ni t \mapsto <T(t)y,x>}\in W(S), \mbox{ for all } y\in Y, x\in X.
$$

Then,
$$
\overline{Y_a}^{\mu(\tau)}\subset \overline{Y_{uds}}^{\mu(\tau)}.
$$
If $\semig{T}{S}$ is equicontinuous with respect to $\mu(\tau),$ then
$\overline{O(y)}^{\mu(\tau)}$ is $\mu(\tau)-$precompact for all $y\in Y_a.$
\end{cor}

\begin{proof}[Proof of Cor. \ref{semigroup-equicontinuous}]
We find a net $\net{y}{\al}{A}\subset Y_{uds}$ with $\mu(\tau)$-limit $y\in \overline{Y_a}^{\mu(\tau)}.$
For every $W, V,U\in \mathcal{U}(0),$ with $V+V\subset U,$ we find a net of finite sets, $\net{F}{\al}{A}$ such that $O(y_{\al})\subset F_{\al}+V.$ From the convergence and equicontinuity, we find $W$ and $\al_0$ such that   $x-x_{\gamma}\in W$  and $T(t)(y-y_{\gamma})\subset V$ for all $\gamma\ge \al_0,$  $t\in S.$ Hence,
$$
O(y)\subset O(y_{\gamma})+V\subset F_{\gamma}+U.
$$
\end{proof}

The result of Fr\'echet \cite{Frechet} for asymptotically almost-periodic functions is a direct consequence.

\begin{cor}\label{aap-consequence}
Let $Z$ be a Banach space, let $\bk{T(t)}_{t\ge0}$ be a bounded semigroup on $Z,$ and let $O(x)$ be relatively compact for all $x\in Z$; then,
$Z_{uds}=Z_a.$
\end{cor}
\begin{proof}
Choose $\tau$ equal to the norm topology and apply Cor. \ref{semigroup-equicontinuous}.
\end{proof}

Using that the Mackey topology $\mu(X^*,X)$ of $\sigma(X^*,X)$ is given by the uniform convergence on the weakly compact set of $X,$ we obtain the following:
\begin{cor} \label{reccurent_equals_unimodular}
If $\overline{\semig{T^{\sun}}{\rep}}^{w^*}$ is Abelian, then
$$
\overline{X^{\sun}_{uds}}^{\mu(X^{*},X)}=\overline{ X_a^{\sun}}^{\mu(X^{*},X)}.$$
\end{cor}

Separability is a further concept that applies. With the pointwise verification of the Abelian structure, we can give the criterion below for a vector to be a member of $X_{uds}^{*},$ which is the second main result of this section, based mainly on a harmonic analysis, the Pettis-measurability criteria and Thm \ref{T-abelian-condition}.
%Editor: Please ensure that the intended meaning has been maintained in the above edit.

\begin{theo}\label{separable-X_uds}
Let $\semig{T}{S}$ be a dual representation, $\tau$ be a representation-compatible topology with $\sigma(X^*,X)\subset \tau,$ and
$$
\bk{S\ni t \mapsto <T(t)y,x>}\in W(S), \mbox{ for all } y\in Y, x\in X.
$$
Then, the following are equivalent for $y\in Y_a$:
\begin{enumerate}
\item $\overline{O(y)}^{\tau}$ is norm separable.
\item $y\in Y_{uds}^{*}$.
\item $O(y)$ is relatively norm compact.
\end{enumerate}
\end{theo}

\begin{proof}[Proof of Theorem \ref{separable-X_uds}]
As $\overline{ac}^{\tau-OT}\fk{\semig{T}{S}}$ is Abelian, $G= P\scST $ is a compact Abelian topological group \cite{Ellis_cont}.
%Editor2: Please note that some text appears to be missing on the line below. Please add any missing information.
The splitting is a consequence of Prop.
From Prop. \ref{start-harmonic-analysis}, we have, for $y\in Y,$
$$
S_{\gamma}y=\int_{G}\overline{\gamma}(S)Sy d\rho(S)\in Y.
$$
By \cite[Cor. 4 pp. 42-43]{DiestelUhl}, we find that this is a Bochner integral, which is an element of $X^*.$ Moreover, because $\sigma(X^*,X)\subset \tau,$ the expression above
%Editor: Please ensure that the intended meaning has been maintained in this edit.
coincides on $X$ with the integral defined in the proof of Theorem \ref{Y_uds=Y_a}; hence, it becomes an element of $Y.$

For $R\in \scST$, we have

\begin{eqnarray*}
RS_{\gamma}y&=&S_{\gamma}Ry=\delta_{Ry}(S_{\gamma})=\int_G\overline{\gamma}(S)SRy d\rho(S) =\int_{G}\overline{\gamma}(S)RSy d\rho(S)=\gamma(R)S_{\gamma}y.
\end{eqnarray*}

Defining $$M=\overline{span}\bk{S_{\gamma}y:y\in Y, \gamma\in \Gamma},$$ we have $M\subset Y_a$.

For $y\in Y_a$ and $q:Y_a\to Y_a/M$ as the quotient map, if $Z=\overline{span}\bk{qGy}$, then by assumption, $(Z,\nrm{\cdot})$ is separable.

Consequently, $(B_{Z^*}, w^*)$ is separable (compact metrizable). We choose $\bk{z^*_n}_{n\in\za}$ dense in $(B_{Z^*}, w^*)$.

By definition, $S_{\gamma}y\in M.$ Consequently, for the sequence of bounded linear functionals
$$
\Funk{U_n}{Y}{\ce}{u}{<qu,z_n^*>,}
$$
due to Bochner integrability, we obtain
$$
0=<qS_{\gamma}y,z_n^*>=\int_G\overline{\gamma}(S)<qSy,z_n^*>d\rho(S)
$$
for all $\gamma\in \Gamma$ and $n\in \za.$ Using $\bk{\gamma}_{\gamma\in \Gamma}$ as an orthonormal basis in $L^2(G,\rho)$,
$$
<qSy,z_n^*>=0 \mbox{ a.e. for all } n\in \za.
$$
Thus, for sets $A_n\subset K$ with $\rho(A_n)=0$, we have
$$
<qSy,z_n^*>=0 \mbox{ for all } S\in G\backslash A_n, n\in \za.
$$
Let $A=\bigcup_{n\in\za}A_n;$ then, $\rho(A)=0$, and
$$
<qSy,z_n^*>=0 \mbox{ for all }S\in G\backslash A, n\in \za.
$$
Using $\bk{z^*_n}_{n\in\za}$ totally on $Z$, we find an $S\in G$ with $qSy=0. $ Consequently, $Sy\in M.$ Because of Part 1 of Theorem \ref{Y_uds=Y_a} and equation (\ref{M-is-T-invariant}), the space $M$ is translation invariant, and for $y\in X_a$, we find that using $G$ as a group on $Y_a,$ there is a $T\in G$ such that $TSy=y$ and, therefore, $y\in M\subset X_{uds}^{\sun}.$

(2)$\Rightarrow$ (3): Let $y \in Y_{uds}$. Then, $y$ is the limit of the linear combinations of the unimodular vectors $\bk{x_i^{n}}_{i=1..m_n,n\in\za}\subset Y_a$,
%Editor2: On the line below, please ensure that the intended meaning has been maintained in this edit.
i.e., those satisfying $Tx_i^n=\la^n_i(T)x_i^n.$
Consequently, $O(x_i^n)$ is norm compact and therefore the orbit of the linear combination.

It follows that if the vectors $\bk{x_n}_{n\in\za}$ have relatively norm-compact orbits and $x_n \to x,$ then $O(x)$ is relatively norm compact. Note that for some constant $C>0$,
$$
\nrm{Tx_n-Tx}\le C\nrm{x-x_n},
$$
which concludes the proof.
\end{proof}

As proposed, we show that the separability of the orbit indicates almost periodicity. Note that if $\semig{T^{\sun}}{\rep}$ is a sun-dual semigroup and $x^{\sun}\in X_{uds}^{\sun},$ then the mapping $\bk{t\mapsto T^{\sun}(t)x^{\sun}}$ is almost periodic. We give a criterion stating when an element in $X_a^{\sun}$ is in $X_{uds}^{\sun}.$

\begin{theo}\label{separable-theorem}
If $x^{\sun}\in X_a^{\sun}$ and $\overline{\semig{T^{\sun}}{\rep}}$ is Abelian, then the following are equivalent:
\begin{enumerate}
\item $\overline{O(x^{\sun})}^{\sigma(X^{\sun},X)}$ is norm separable.
\item $x^{\sun}\in X_{uds}^{\sun}$.
\item $O(x^{\sun})$ is relatively norm compact.
\end{enumerate}
\end{theo}

\begin{cor}
If $\overline{O(x^{*})}^{\sigma(X^{*},X)}$ is norm separable for all $x^{*}\in X_a^{*}$ and $\scST$ is Abelian, then
$$X_a^{*}= X_{uds}^{*}.$$
\end{cor}

\begin{cor}
If $X^{*}$ is norm separable and $\scST$ is Abelian, then
$$
X^{*}_{a}=X^{*}_{uds}.
$$
\end{cor}

\section{Ideal Theory}
In this section we recall the main results on right semitopological semigroups, which are used to obtain the splittings and the group (sub)structures. As mentioned before, in this section $S$ is only right(left)-semitopological and not necessarily Abelian.
\begin{defi}
A subset $\scJ$ of a semigroup $\scS$ is called a right (left) ideal of $\scS$ if
$$\scJ\scS:=\bk{xs:x\in\scJ,s\in \scS}\subset \scJ \ \  (\scS\scJ:=\bk{sx:x\in\scJ,s\in \scS}\subset \scJ).
$$
A subset is called an ideal if it is both a right and a left ideal.
\end{defi}

\begin{theo}[\cite{Ellis_cont}]
Every compact right [left] topological semigroup has an idempotent.
\end{theo}
\begin{defi} [ {\cite[p. 12]{RuppertLNM}}]
The set of idempotents in a semigroup $S$ is denoted $E(S).$ We define relations $\le_L$ and $\le_R$ on E(S) by
\begin{eqnarray*}
e\le_L f &\mbox{ if } & ef=e, \\
e\le_R f &\mbox{ if } & fe=e.
\end{eqnarray*}
If $e$ and $f$ commute, then we omit the indices L and R.
\end{defi}
\begin{defi}
Let $(A,\le)$ be a set with a transitive relation.
Then, an element $a$ is called $\le-$maximal [$-$minimal] in $A$
if for every $a^{\prime}\in A$, $a\le a^{\prime}$ implies $a^{\prime} \le a$ [$a^{\prime} \le a$ implies $a\le a^{\prime}$].
%Editor: Please be consistent in using either parentheses or brackets; refer to the style guide of the journal of interest for instructions.
\end{defi}
Recalling \cite[p. 14]{RuppertLNM}, we have the following:
\begin{theo}
Every compact right topological semigroup contains $\le_L-\mbox{maximal}$ and $\le_R-\mbox{minimal}$ idempotents.
\end{theo}

\begin{theo}[{\cite[p. 21]{RuppertLNM}}] \label{right_topo_semigroups}
For an idempotent $e$ in a compact right topological semigroup $S$, the following statements are equivalent:
\begin{enumerate}
\item $e$ is $\le_R-$minimal in $E(S)$.
\item $e$ is $\le_L-$minimal in $E(S)$.
\item $eS$ is a minimal right ideal of $S$.
\item $eSe$ is a group, and $e$ is an identity in $eSe$.
\item $Se$ is a minimal left ideal of $S$.
\item $SeS$ is the minimal ideal of $S$.
\item S has a minimal ideal $M(S)$ and $e\in M(S).$
\end{enumerate}
\end{theo}

\section{Several Lemmas}

\begin{pro} \label{interchanged1}
Let $S$ be Abelian, $f:S\to\ce$ be E.-wap and $\net{t}{\la}{\Lambda},
\net{s}{\gamma}{\Gamma}\subset S.$
Then, we may pass to subnets $\bk{s_{\gamma_{\alpha}}}_{\alpha\in A}$ and $\bk{t_{\la_{\beta}}}_{\beta\in b}$ such that the iterated limits
\begin{eqnarray*}
\nu&=& \netlim{\alpha}{A}\netlim{\beta}{B} f(t_{\la_{\beta}}+s_{\gamma_{\al}}) \mbox{ and } 
\mu= \netlim{\beta}{B}\netlim{\al}{A}f (t_{\la_{\beta}}+s_{\gamma_{\al}})
\end{eqnarray*}
exist, and we have $\nu=\mu.$
\end{pro}
\begin{proof}
Because $f$ is E.-wap, $\bk{f_{t_{\la}}}_{\la\in\Lambda}$ is rel.~$\sigma(C_b(S),C_b(S)^*)-$compact,
$\bk{\delta_{s_{\gamma}}}_{\gamma\in\Gamma}$ is relatively $w^*$-compact, and we may pass to convergent subnets. Using $f(t_{\la}+s_{\gamma})=\delta_{s_{\gamma}}f_{t_{\la}}$, we find that the iterated limits exist and that they are equal.
\end{proof}

Next, we recall the consequence of \cite[Cor. 2 (a), p. 127]{Jarchow}, and we have the following:
\begin{pro} \label{tau-separated}
Let $E\subset F\subset X^*$ and $\tau$ be a locally convex topology on $X^*$, where $\sigma(X^*,X)\subset \tau \subset \nrm{\cdot}.$
If no vector of $E$ can be separated from $F$ by a $\tau-$continuous functional, then
$\overline{E}^{\tau}=\overline{F}^{\tau}.$
\end{pro}
\begin{proof}
If $\overline{E}^{\tau}\not =\overline{F}^{\tau},$ then there exists an $x\in \overline{F}^{\tau} \backslash \overline{E}^{\tau}$ and a $\tau-$continuous functional $\phi$ such that $\phi_{|\overline{E}^{\tau}}=0$ and $\phi(x)=1.$ By definition, we have for net $\net{x}{\la}{\Lambda}\subset F$ the $\tau$-convergence $x_{\la}\to x.$ Moreover, we find a subnet that has no intersection with $\overline{E}^{\tau}.$ The continuity $\phi$ leads to an element $x_{\la_0}$ with $\phi(x_{\la_0})>1/2,$ which illustrates the contradiction.
\end{proof}

We start with the main lemma, which is applied in several ongoing circumstances.

\begin{lem} \label{general-metric-in-X-star}
Let $K\subset X^*$ be a $w^*-$compact and norm-separable set. Then, $(K,w^*)$ is compact metrizable, and the metric is given by
$$
d(x,y)=\sum_{n=1}^{\infty}\frac{\btr{<z_n,x-y>}}{1+ \btr{<z_n,x-y>}}
$$
for a sequence $\seq{z}{n}\subset X.$
\end{lem}

\begin{proof}
By the assumption, there is a sequence $\bk{y^{*}_n}_{n\in\za}$ such that
$$
K\subset \overline{\bk{y^{*}_n}_{n\in\za}}^{\nrm{\cdot}}.
$$
If $Z=\overline{span}\bk{y^{*}_n}=\overline{\bk{x^{*}_n}_{n\in\za}}\subset X^{*},$
then $B_{Z^*}$ is compact metrizable,
%Editor2: On the line below, please ensure that the intended meaning has been maintained in this edit.
therefore separable, where
$K\subset \overline{Z}^{w^*}$.

Hence, let
$\overline{\bk{z^*_n}_{n\in\za}}^{w^*}$ be $w^*-$dense in $B_{Z^*}.$
Let $x_n^{**}\in X^{**}$ be the sequence of extensions of $z_n^*:Z\to \ce.$ Note that for the natural embedding
$j_X:X\to X^{**},$  we have $\overline{j_XB_X}^{\sigma(X^{**},X^{*})}=B_{X^{**}}$; compare \cite[p. 424]{DS}. We define the sequence of open sets,
$$
U_{k,l}:=\bk{x^{**}\in X^{**}: \btr{x^{**}(x_m^{*})}<\frac{1}{k}, 1\le m \le l}.
$$
Then, $U_{k,l}$ is $w^*-$open zero-neighborhood. Consequently, for all $k,l\in \za,$ we find $x_n^{k,l}\in B_X$
such that $j_Xx^{k,l}_n-x_n^{**}\in U_{k,l}.$ Letting $H=\overline{span}\bk{x_n^{k,l}:k,l,n\in\za},$ we claim that
$(Z,\sigma(Z,H))$ is Hausdorff.
If $z\in Z$ and $z(x_m^{k,l})=0 $ for all
%Editor2: On the line below, please ensure that the intended meaning has been maintained in this edit.
$m,k,l\in \za$, by definition,
for all $\ep >0$, we find an $x_n^{*}$ such that $\nrm{x_n^{*}-z} \le \ep.$

\begin{eqnarray*}
\btr{z_m^*(x_n^{\sun})}&=&\btr{x^{\sun*}_m(x^{\sun}_n)} \le \btr{(x^{\sun*}_m-x_m^{k,l})(x^{\sun}_n)}+\btr{x_m^{k,l}(x^{\sun}_n)} \\
&\le& \frac{1}{k} +\btr{x_m^{k,l}(x^{\sun}_n)},
\end{eqnarray*}
for all $0\le l\le n.$ Furthermore,
\begin{eqnarray*}
\btr{x_n^{k,l}(x^{\sun}_m)}&\le& \btr{x_m^{k,l}(x^{\sun}_n-z)}+0 \\
&\le&\nrm{x_n-z} \le \ep.
\end{eqnarray*}
Hence, $x^{\sun}_n=0,$; therefore, $z=0.$ With $\bk{\btr{<x_n^{k,l},\cdot>}:k,l,n\in\za}$, we found a countable set of seminorms, which induce the Hausdorff metric on $K,$ which is weaker than the $w^*-$topology.
\end{proof}

Next, we recall a lemma given as \cite[Lemma 4]{Namioka}. Due to its central role and the rudimentary arguments given in the original study, a more detailed argument is given below.

\begin{lem}[\cite{Namioka}]   \label{separability-lemma}
%Editor2: On the line below, please ensure that the intended meaning has been maintained in this edit.
  Let $K$ be a compact Hausdorff space and
  $M \subset C(K)$ be closed bounded and separable with respect to the pointwise topology. Then, $M$ is norm separable in $C(K).$
\end{lem}

\begin{proof}
We start by defining an equivalence on $K.$
$$
t\thicksim s \ :\Leftrightarrow \ \sup_{g\in M}\nrm{g(t)-g(s)}=0.
$$
The rules for the relations are simply given; let $\tilde{K}:=K/\thicksim$ be a quotient with the quotient topology and $q:K\to \tilde{K}$ be the quotient map.
Considering $M\subset B_{C(K)}(0,L)$ and $D\subset M$ as a countable dense set, define
$$
\Funk{U}{K}{\Pi_{g\in D}[-L,L]}{t}{g(t)}.
$$
Then, the product is a compact metric, and $\vp(K)$ is a compact subset. Moreover, $U= \tilde{U}\circ q,$ with
$$
\Funk{\tilde{U}}{\tilde{K}}{\Pi_{g\in D}[-L,L]}{t}{g(t).}
$$
Similarly, we have that for $f\in M,$ $f= \tilde{f} \circ q.$  Due to the continuity of $f,$ we have that $\tilde{f}$ is continuous \cite[Thm. 9, p. 95]{Kelley}. The continuity of $q$ and the compactness of $K$ give us that $\tilde{K}$ is compact.
It is simple to verify that $\tilde{U}$ is injective. Hence, $\tilde{K}$ and $U(K)$ are homeomorphic, and
$$
\hat{f}:=\tilde{f}\circ \tilde{U}^{-1} \mbox{ or } f(t)=\hat{f}\circ U(t).
$$
The first identity proves that $\hat{f}$ is continuous.
%Editor2: On the line below, please ensure that the intended meaning has been maintained in this edit.
Additionally, because $\nrm{g(t)}=\nrm{\tilde{g}(U(t))}$,
pointwise denseness gives $\nrm{f(t)}=\nrm{\hat{f}(U(t))}.$ Therefore, $\nri{f}=\nri{\hat{f}}.$ Consequently,
$$
\Funk{V}{M}{C(U(K))}{f}{\hat{f}}
$$
is an isometry, and the norm-separability of $C(U(K))$ implies that $T(M)$ is norm separable; therefore, $M.$
\end{proof}

With the aim of obtaining the translation operator on $BUC(S;X)$ as a dual operator, the following proposition points to a sufficient condition.
\begin{pro}
Let $S$ be a Borel measure space and the measure $\mu$ be such that for all $f\in BUC(S),$
$$
\sup_{g\in B_{L^1(S}}\int_{S}\btr{fg}d\mu=\sup_{t\in S}\btr{f(t)}.
$$
Then, using
$L^1(S,X))^*=(L^1(S)\otimes_{\pi} X)^*=L(L^1(S),X^*),$ for
$$
\Funk{\iota}{BUC(S,X^*)}{L(L^1(S),X^*)}{f}{\bk{ g \mapsto  \bk{ x\mapsto \int_{S}<g(r)\otimes x,f(r)>d\mu(r)}}},
$$
we have $\nri{\iota(f)}=\nri{f}.$ The duality of $(L^1(S,X),BUC(S,X^*))$ is given by
$$
<g,f>=\int_S<g(r),f(r)>d\mu(r).
$$

If in addition $S$ is a subsemigroup of a group, and $\mu$ is an invariant measure. For $t\in S$ the translation operator 
$$
\Funk{T(t)}{BUC(S,X^*)}{BUC(S,X^*)}{f}{\bk{s\mapsto f(t+s)}, }
$$
is a restriction of a dual operator to
$$
\Funk{V(t)}{L^1(S,X)}{L^1(S,X)}{f}{\bk{ s\mapsto \left\{
\begin{array}{rcl} 
f(s-t) &:& s\in  t+S \\ 
0 &:& otherwise 
\end{array} \right. }.}
$$
\end{pro}

\section{Applications}
After the given general approach, we show to which cases this may apply. Moreover, some examples are given, which show the restrictions.

\subsection{Ergodic Results}
To obtain ergodic results one may apply the general theory of \cite{Gerlach}. 
In this section an application to the sun-dual semigroup is given, using the compactness of the convex semigroup 
$\overline{co}^{w^*}\fk{\semig{T^{\sun}}{\rep}}.$ 
It is shown that $N(A^{\sun})$ is complemented in $X^{\sun}.$
%%%Quality Control Editor - Please ensure that the intended meaning has been maintained in the above edit.

An application of $\scTOS$ is found in \cite{Gerlach}, where the theory of norming dual pairs is discussed. Note that $(X,X^{\sun},<\cdot,\cdot>)$ is such a norming dual pair. 
We recall that 
$$
C^{\sun}(r):=\frac{1}{r}\int_0^rT^{\sun}(s)ds \in \overline{co}^{w^*}\fk{\semig{T^{\sun}}{\rep}}
$$
and
$$
(T^{\sun}(t)-I)C^{\sun}(r)x^{\sun}\to 0 \mbox{ in } \nrm{\cdot}.
$$
Thus, \cite[Lemma 4.5]{Gerlach} leads to the following.
\begin{cor} Let $\bk{T(t}_{t\ge 0}$ be a $C_0-$semigroup with generator $A$. Then, we have, for the mean of the dual semigroup and an appropriate net $\net{t}{\la}{\Lambda}$,
$$
\sigma(X^{\sun},X)-\netlim{\la}{\Lambda}C^{\sun}(r_{\la})x^{\sun} \in N(A^{\sun}),
$$
and
$
\topoks-\netlim{\la}{\Lambda}C^{\sun}(r_{\la})=Q^{\sun}$ is a projection onto $N(A^{\sun}).$ 
\end{cor}
\begin{proof}
By \cite[Lemma 4.5]{Gerlach}, we have $Q^{\sun}x^{\sun}\in N(A^{\sun}).$ Let $x^{ \sun}\in N(A^{\sun})$; then, 
$C(r)x^{\sun}\equiv x^{\sun}=Q^{\sun}x^{\sun}.$ It remains to prove that $Q^{\sun}Q^{\sun}=Q^{\sun}.$ 
If $x^{\sun}\in X^{\sun}$ and $Q^{\sun}x^{\sun}=y^{\sun}\in N(A^{\sun})$, then 
$$
Q^{\sun}Q^{\sun}x^{\sun}=Q^{\sun}y^{\sun}=y^{\sun}=Q^{\sun}x^{\sun},
$$
which concludes the proof.
\end{proof}

\subsection{Application to general Banach spaces}
In this section we show that in counter to \cite{Witz} the compatification of a bounded $C_0-$semigroup resides in a smaller operator space than $L(X,X^{**}).$

We start by considering the Banach algebra,
\begin{equation}
\LT(X^{\sun}):=\bk{T\in L(X^{\sun}): T^*(X^{\sun\sun})\subset X^{\sun\sun}}\subset L(X^{\sun},X^*),
\end{equation}
the operator space
\begin{eqnarray}
\LT(X,X^{\odot\odot})&:=&\bk{U\in L(X,X^{\odot\odot}): \ U^*(X^{\odot})\subset X^{\odot}, U^{\sun*}(X^{\sun\sun})\subset X^{\sun\sun}} \\
&\subset&  L(X,X^{\sun*}), \nonumber
\end{eqnarray}
and
$$
\Funk{\eta^{\sun}}{\LT(X,X^{\sun\sun})}{\LT(X^{\sun})}{V}
{\bk{x^{\sun}\mapsto \eta^{\sun}V: x\mapsto <x,V^*_{|X^{\sun}}x^{\sun}>(=<Vx,x^{\sun}>)}}
$$
endowing $\LT(X^{\sun})$ and $\LT(X,X^{\odot\odot})$ with their relative $\topoks$ topology. Noting that for general Banach spaces \cite[$(2^{\prime})$, p. 154]{Koethe2}
$$
L(X,X^{\sun*})\cong (X^{\sun}\otimes_{\pi}X)^* \cong L(X^{\sun},X^*),
$$
in the next lemma, it is proven that if the operator spaces depending on the semigroup $\bk{T(t)}_{t\ge0}$ are considered, then $*$ may be replaced by $\sun.$

\begin{lem}\label{eta_sun-algebra_isomorphism}
Let
$$
\Funk{\eta^{\sun}}{\LT(X,X^{\sun\sun})}{\LT(X^{\sun})}{V}
{\bk{x^{\sun}\mapsto \eta^{\sun}V: x\mapsto <x,V^*_{|X^{\sun}}x^{\sun}>(=<Vx,x^{\sun}>)}}
$$
Then, $\eta^{\sun}(T(t))=T^{\sun}(t),$
$\eta^{\sun}$ is an isomorphism and
%Editor: Please ensure that the intended meaning has been maintained in this edit.
$\topoks-\topoks$-continuous, and $\eta(V\circ U)=U^{\sun}V^{\sun}.$ By the use of $(\eta^{\sun})^{-1},$ the operator space $\LT(X,X^{\sun\sun})$ becomes a Banach algebra, with
$$
\Funk{U\circ V}{X}{X^{\sun\sun}}{x}{\bk{ x^{\odot }\mapsto <Vx,U^{\odot}x^{\odot}>}},
$$
\end{lem}
\begin{proof}[Proof of Lemma \ref{eta_sun-algebra_isomorphism}]
To prove that $\eta^{\sun}$ is surjective, let $U\in \LT(X^{\sun})$ and $V:=U^*_{|X^{\sun\sun}}\in L(X^{\sun\sun}),$ and if $j:X\to X^{\sun\sun}$ denotes the natural embedding, we can claim $\eta^{\sun}(\bk{x\mapsto V(jx)})=U.$
Because
$$
<x,Ux^{\sun}>=<jx,Ux^{\sun}>=<V(jx), x^{\sun}>=<x,\eta^{\sun}(Vj)x^{\sun}>,
$$
which proves the claim. To verify the injectivity, note that
$$
\nrm{\eta(U)}=\sup_{x\in B_X}\sup_{x^{\sun}\in B_{X^{\sun}}}\btr{<Ux,x^{\sun}>}=\nrm{U}.
$$

Note that $V, U\in \LT(X,X^{\sun\sun});$ then,
\begin{eqnarray*} \label{algebra-structure}
<x,\eta^{\sun}(V)\eta^{\sun}(U)x^{\sun}>
&=& <x,V^{\sun}_{|X^{\sun}}U^{\sun}_{|X^{\sun}}x^{\sun}>   \\
&=&<Vx,U^{\sun} x^{\sun}> = <(U\circ V)x,x^{\sun}>.
\end{eqnarray*}
\end{proof}

As $\semig{T^{\sun}}{\rep}$ is a bounded subset of $L(X^{\sun},X^*)$, its $\topoks$ closure is compact in $L(X^{\sun},X^*).$ Using
$$\semig{T}{\rep}\subset (\eta^{\sun})^{-1}(\overline{\semig{T^{\sun}}{\rep}}^{w^*})$$
is densly contained, we may consider the compactification in a smaller operator space.

\begin{theo} \label{Fundamental_Theorem}
Let $\scST:=\bk{T(t)}_{t\ge0}$ and $\scSTS:=\bk{T^{\sun}(t)}_{t\ge0}.$ Then,
$$\overline{\scSTS}^{\topowsOT}\subset L(X^{\sun}),$$
and  for 
$$
\scSTO=(\eta^{\sun})^{-1}\fk{\overline{\semig{T^{\sun}}{\rep}}^{w^*}}$$
holds.

\begin{enumerate}
\item  $\scST \subset \scSTO\subset\scTO \subset \scATO\subset L(X,X^{\sun\sun})$,
\item $\scSTO$ is compact and a right semitopological semigroup,
\item $\scTO$ is convex, compact and a right semitopological semigroup, and
\item $\scATO$ is absolutely convex, compact and a right semitopological semigroup
\end{enumerate}
with respect to the $\topoks-$topology.\\
\end{theo}

A connection with the deLeeuw--Glicksberg theory note is provided as follows:
\begin{theo} \label{DG-1}
If $\bk{T(t)}_{t\in\rep}$ is E.-wap, then
$X=X_{ap}\oplus X_0$, with a projection $\funk{V}{X}{X}$ satisfying $V(X)=X_{ap}.$ For the dual semigroup, we have
$X_a^{\sun}=\Xsr, \ X_0^{\sun}=X_{fl}^{\sun},$ with $X^{\sun}=\Xsr\oplus X_{fl}^{\sun},$ with a projection $\funk{P^{\sun}}{X^{\sun}}{X^{\sun}}$ satisfying $P^{\sun}(X^{\sun})=X^{\sun}_a.$ In this setting, we have $P^{\sun}=\eta^{\sun}(V),$ and the minimal idempotent is unique.
\end{theo}
\begin{proof}[Proof of Theorem \ref{DG-1}] It suffices to verify that $P^{\sun*}(X)\subset X. $
By Prop. \ref{first-results}, and Prop. \ref{embed-DeGli-theory}   we find that $X=X_{ap}\oplus X_0.$ Let $V$ be the corresponding projection and $V^{\sun}:=\eta(V).$ Furthermore, let $X^{\sun}=X^{\sun}_a\oplus X^{\sun}_0,$ and let $P^{\sun}$ be the corresponding minimal idempotent. We define $P:=\eta^{-1}(P^{\sun}).$ Then,
\begin{eqnarray*}
<x,V^{\sun}V^{\sun}x^{\sun}>&=&<Vx,V^{\sun}x^{\sun}>=<V\circ V x,x^{\sun}> \\
&=&<Vx,x^{\sun}>=<x,V^{\sun}x^{\sun}>,
\end{eqnarray*}
and for $P$, we have
\begin{eqnarray*}
<(P\circ P)x,x^{\sun}>&=& <Px,P^{\sun}x^{\sun}>=<x,P^{\sun}P^{\sun}x^{\sun}>\\
&=&<x,P^{\sun}x^{\sun}>=<Px,x^{\sun}>.
\end{eqnarray*}
Hence, $P$ and $V^{\sun}$ are idempotents in $\scSTO$ and $\scSTOS$.

By Theorem \ref{right_topo_semigroups}, we have that $V$ is minimal using the fact that $\scSTO$ is an (Abelian) group on $X_{ap}=VX$ and $P^{\sun}$ is a minimally chosen idempotent.
Moreover, given that $\scSTO$ is Abelian, we find that $VP$ is an idempotent with $V(VP)=VP$; hence, $VP=V.$
%Editor: Please ensure that the intended meaning has been maintained in the above edit.
Similarly, we obtain from $P^{\sun}(P^{\sun}V^{\sun})=P^{\sun}V^{\sun};$ hence, $P^{\sun}=P^{\sun}V^{\sun}$ because of its minimality.
This result leads to
\begin{eqnarray*}
<x,\eta(V)x^{\sun}>&=&<x,\eta(V\circ P)x^{\sun}>=<x,P^{\sun}V^{\sun}x^{\sun}>\\
&=&<x,P^{\sun}x^{\sun}> =<x,\eta(P)x^{\sun}>.
\end{eqnarray*}
In the first line, the $V$ left minimal is used, and in the
second, the $P^{\sun}$ left minimal is used.%Editor: Please ensure that the intended meaning has been maintained in the above edit.
Because $\eta$ is injective, we have that $V=P$ and $\scSTO(X)\subset X$ by the Eberlein-weak almost periodicity; we thus conclude that $P(X)=V(X)\subset X.$
\end{proof}

It is always a question when an orbit that is a continuous image of $\re$ or $\rep$ is metric or at least separable. In this section, we give some results, which are split into the general case of sun-dual $C_0-$semigroups and the special case of the translation semigroup on bounded uniformly continuous functions.
\subsection{Separability and metrizability of the orbits of general bounded $C_0-$semigroups}
Throughout this section we discussing $C_0-$semigroups, therefore $S$ is assumed to be either $\rep,$ or $\re,$ if a $C_0-$group is in discussion.

\begin{cor} \label{metric-orbits-T-sun-of-t}
Let $\semig{T}{\rep}$ be a bounded $C_0-$semigroup, $x^{\sun}\in X^{\sun},$ and $\overline{\bk{T^{\sun}(t)x^{\sun}}_{t\in\rep}}^{w^*}$ be norm-separable in $X^{\sun}.$ Then, $\fk{\overline{\bk{T^{\sun}(t)x^{\sun}}_{t\in\rep}}^{w^*},w^*}$ is compact metrizable.
\end{cor}
\begin{proof}
%Editor2: On the line below, please ensure that the intended meaning has been maintained in this edit.
Apply Lemma \ref{general-metric-in-X-star} to a compact and separable $\overline{\bk{T^{\sun}(t)x^{\sun}}_{t\in\rep}}^{w^*}\subset X^{\sun}$.
\end{proof}

\begin{cor} \label{metric-orbits-T-sun--restricted-of-t}
Let $\semig{T}{\rep}$ be a bounded $C_0-$semigroup, $Y\subset X^{\sun},$ $y\in Y,$ and $\bk{T^{\sun}(t)_{|Y}}_{t\in\rep}$ a dual representation with a representation compatible topology $\tau.$ If $\overline{\bk{T^{\sun}(t)y}_{t\in\rep}}^{\tau}$ is $\tau-$compact and norm-separable in $Y$, then $\fk{\overline{\bk{T^{\sun}(t)y}_{t\in\rep}}^{\tau},\tau}$ is compact metrizable.
\end{cor}
\begin{proof}
Due to the $\tau-$compactness of $\overline{\bk{T^{\sun}(t)y}_{t\in\rep}}^{\tau},$ recalling the definition of the dual semigroup representation, the $w^*$-topology coincides with the stronger topology $\tau.$
\end{proof}

Next, we apply the previous result to obtain a general metrizability criterion for bounded $C_0-$semigroups on a Banach space $X.$ Let $j_X:X\to X^{\sun\sun}$ be the canonical embedding.

\begin{cor} \label{metric-orbit-T-of t}
Let $\semig{T}{\rep}$ be a bounded $C_0-$semigroup, $x\in X,$ and
$\overline{\bk{T^{\sun\sun}(t)j_Xx}_{t\in\rep}}^{\sigma(X^{\sun\sun},X^{\sun})}$ be separable in $X^{\sun\sun}.$ Then,
$\fk{\overline{\bk{T(t)x}_{t\in\rep}}^{\sigma(X,X^{\sun})},\sigma(X,X^{\sun})}$ is metrizable.
\end{cor}
\begin{proof}
Use the fact that $\bk{j_XT(t)x}_{t\in\rep}\subset \overline{\bk{T^{\sun\sun}(t)j_Xx}_{t\in\rep}}^{\sigma(X^{\sun\sun},X^{\sun}},$ and apply Lemma \ref{metric-orbits-T-sun-of-t}.
\end{proof}
In view of asymptotics, we define the classes of vectors below.

\begin{defi}
Let $M \subset X^*$ be a closed and separating space.
\begin{enumerate}
\item A vector $x\in X$ is called $M-$weakly almost periodic, if for every $x \in M,$ there is a $g \in AP(\re)$ such that $\bk{t\mapsto x^*T(t)x}=g_{|\rep}(t).$ A $C_0-$semigroup
$\semig{T}{\re}$ is called $M-$weakly almost periodic on $X$ if every vector in $X$ is an $M-$weakly almost periodic vector.
\item  A vector $x\in X$ is called $M-$weakly E.-wap (M-w-Ewap) if for every $x^* \in M,$ the mapping $\bk{t\mapsto x^*T(t)x}\in W(\rep.)$. A $C_0-$semigroup
$\semig{T}{\re}$ is called $M-$weakly E.-wap (M-w-Ewap) on $X$ if every vector in $x$ is M-w-Ewap.
\end{enumerate}
\end{defi}

\begin{lem} \label{X-star-Ewap Orbit-metric}
If $\bk{T(t)}_{t\ge 0}$ is a bounded $C_0-$semigroup and $x\in X$ is a $X^*$-w-Ewap vector, then $(\overline{\bk{T(t)x}_{t\ge 0}}^{w},w)$ is metrizable.
\end{lem}

\begin{proof}
Apply the fact that for $K$ weakly compact, $j_XK\subset j_XX,$ and apply Cor. \ref{metric-orbit-T-of t}.
\end{proof}

Recalling Lybich, Yu.I. and Kadets \cite[Thm.2]{KadetsStrongWeak},
%Editor2: On the line below, please ensure that the intended meaning has been maintained in this edit.
they propose conditions that Banach spaces have to fulfill, which state that a weakly almost periodic semigroup ${t\mapsto x^*(T(t)x)}\in AP(\re)$ for all $x^*\in X^*$ must be almost periodic.

\begin{theo}\cite[Thm. 2]{KadetsStrongWeak}
Let a Banach space $X$ have the following property: the weak* sequential closure of each of its separable subspaces Y in the second conjugate space $Y^{**}$ is separable. Then, each $X^*-$weakly almost periodic group acting on $X$ is almost periodic.
\end{theo}

\begin{proof}
Consider the dual semigroup $\bk{T^{\sun\sun}(t)}_{t\ge 0}$ with $jT(t)x=T^{\sun\sun}(t)jx,$ which is an extension to a dual semigroup. By Prop. \ref{weak-almost-periodicity-implies-X-a}, $jx\in X^{\sun\sun}_a,$ and we are in the situation of Theorem 3.12. It remains to verify that the $w^*-$closure of the orbit is sequentially separable in $X^{\sun\sun}.$ Lemma \ref{metric-orbits-T-sun-of-t} implies that the closure and the sequential closure of the orbit coincide. The proof concludes using the assumption that $\overline{O(x)}^{w^*}=\overline{O(x)}^{seq-w^*}\subset \overline{jZ}^{seq-w^*}$ is separable when $Z$ is separable. Choose $Z=\overline{span}\bk{T(\re)x}.$
\end{proof}

\begin{cor}
Let $X$ be a Banach space and $\semig{T}{\rep}$ be a bounded $C_0-$semigroup. If $X^{\sun}$ is separable, then for every
$X-$weakly almost periodic vector $x^{\sun},$ the function  $\bk{t\mapsto T^{\sun}(t)x^{\sun}}=g_{|\rep}(t),$ with $g\in AP(\re,X^{\sun}).$
\end{cor}

\subsection{Vector-valued functions and their orbits}
We consider the space of bounded uniformly continuous functions, which is the space of the translation semigroup or group. A special type of function is for bounded $C_0-$semigroups, $f(\cdot):=S(\cdot)x.$ This view moves behavior from the translation semigroup to general $C_0-$semigroups.

We start by verifying that the translation semigroup on the bounded uniformly continuous functions is in a sense a restriction semigroup, which is a notion introduced in a previous section. In doing so, topologies come into play. We start with the-$w^*$-compact open topology.

\begin{lem} \label{pointwise-weak-star-stronger-than-L1}
On bounded sets $A\subset BUC(\re,X^*)$, the $w^*$-compact-open topology is stronger than $\sigma(BUC(\re,X^*),L^1(\re,X)).$
\end{lem}

\begin{proof} Let $\tau_{w^*-co}^{\mathcal{A}}$ be the bounded topology due to $weak^*$-compact-open convergence that comes with $A_n:=2^nB_{BUC(\re,X^*)}$ $\mathcal{A}:=\seq{A}{n}.$ Furthermore, let $\net{f}{\gamma}{\Gamma}\subset BUC(\jz,X^*)$ be a bounded and $\tau_{w^*-co}^{\mathcal{A}}$-convergent net with the limit $f.$
Because $g \in L^1(\jz,X),$ and $\ep >0,$ we find that $\bk{K_i}_{i=1}^n \subset \mathcal{P}(\re)$
and $\bk{x_i}_{i=1}^n \subset X,$ so that for $\vp_:=\chi_{K_i},$ we have $\nrm{ g - \sum_{i=1}^n \vp_i x_i}_1 < \ep.$ Define $\vp:=\sum_{i=1}^n \vp_i x_i.$

Then, $\bk{f_{\gamma}-f}_{\gamma\in\Gamma}$ is bounded and
$$
\btr{<f_{\gamma}-f,g>}\le\btr{<f_{\gamma}-f,\vp>}+\btr{<f_{\gamma}-f,g-\vp>}.
$$
Hence, it remains to prove convergence on $L^1(\jz)\otimes_{\pi} X$. However, for $x\in X,$
$$
\Funk{T_x}{(BUC(\jz,X^*),\tau_{w^*-co}^{\mathcal{A}})}{(BUC(\jz),\tau_{co})}{f}{<x,f>}
$$
is continuous, and
$$
\Funk{i_T}{(BUC(\jz),\tau_{co})}{(BUC(\jz),\sigma(BUC(\jz),L^1(\jz))}{f}{f.}
$$

Hence, we have the mapping
$$
\funk{i_A \circ T_x}{(BUC(\jz,X^*),\tau_{w^*-co}^{\mathcal{A}})}{(BUC(\jz),\sigma(BUC(\jz),L^1(\jz)),}
$$
and therefore,
$$
\funk{id}{(BUC(\jz,X^*),\tau_{w^*-co}^{\mathcal{A}})}{(BUC(\jz,X^*),\sigma(BUC(\jz,X^*),L^1(\jz,X))}.
$$
\end{proof}

After this comparison, we are ready to verify the first compactness result,
%Editor2: On the line below, please ensure that the intended meaning has been maintained in this edit.
which introduces the provided theory.

\begin{lem} \label{BUC-X-star-is-dual-representation}
Let $X$ be a Banach space and $f\in BUC(\rep,X^*)$; then,
$$
\overline{ac}^{\sigma(BUC(\re,X^*),L^1(\rep,X))}\bk{f_t:t\in \rep} \subset BUC(\rep,X^*)
$$
is compact, and the translations are dual representations on $BUC(\re,X).$
Therefore, independent of the a.p., we have (a not necessarily unique splitting)
$$ BUC(\rep,X^*)=BUC(\rep,X^*)_a\oplus BUC(\rep,X^*)_0,$$
and this splitting is nontrivial because
$$
AP(\re,X^*)\subset BUC(\re,X^*)_a, \mbox{ and } \ W_0(\re,X^*)\subset BUC(\re,X^*)_0.
$$
\end{lem}
\begin{proof}
Let $\bk{t^i_{\gamma}}_{\gamma\in\Gamma,i\in\za}\subset \re,$  $\bk{\la^i_{\gamma}}_{\gamma\in\Gamma,i\in\za}\subset \ce,$ and $\net{n}{\gamma}{\Gamma}\subset \za,$ with $\sum_{i=1}^{n_{\gamma}}\btr{\la^i_{\gamma}}=1.$

Using
$$
\bk{\bk{\sum_{1}^{n_{\gamma}}\la^i_{\gamma}f(t+t_i^{\gamma})}_{t\in\re}}_{\gamma\in\Gamma}\subset \Pi_{t\in\re} \
\fk{\overline{ac}^{w^*}\bk{f(t):t\in\re},w^*},
$$
by Tychonov's theorem, we find a subnet converging pointwise $w^*$ to a bounded function element $\tilde{g},$
and the lower semicontinuity of the $w^*$-topology of $X^*$ gives us that $\tilde{g}$ is uniformly continuous with the modulus of $f.$ Hence, by equicontinuity, this net is $\tau_{w^*-co}$-convergent.
The previous lemma \ref{pointwise-weak-star-stronger-than-L1} gives
$$
\sigma(BUC(\re,X^*),L^1(\re,X))-\netlim{\gamma}{\Gamma}\sum_{1}^{n_{\gamma}}\la^i_{\gamma}f(t+t_i^{\gamma})=\tilde{g}\in \overline{acO(f)}^{w^*}.
$$
Consequently, $\tilde{g}\in L^1(\re,X)^{\sun}\cap BUC(\re,X^*),$ and Prop. \ref{first-results}	concludes the proof.

\end{proof}

\subsection{The setting for vector-valued functions}

\begin{lem} \label{X-sep-implies-BUC-weak-star-metrizable}
If $X$ is separable, then $(BUC(\re,X),\sigma(BUC(\re,X),L^1(\re,X^*)),$ and every \\ $\sigma(BUC(\re,X),L^1(\re,X^*)-$compact subset is metric.
\end{lem}
\begin{proof}
Due to the compact metrizability of $B_{X^*},$ let $\seq{x^*}{m}$ be a dense sequence. Additionally, choose $\seq{vp}{n}$ dense in $L^1(\re).$
If $f\in BUC(\re,X)$ is such that $\int_{\re}x_m^*(f(t))\vp_n(t)dt=0$ for all $n,m\in \za,$ then, by the denseness of $\seq{\vp}{n},$ we have $x_m(f(t))=0$ for $m\in \za,$ and therefore $f=0.$ Let $\tau_{pw-d}$ be the topology induced by this countable set of elements; then, every compact subset becomes metric with respect to this coarser topology.
\end{proof}

\begin{pro} \label{BUC_closed_in_Dual}
Let $\bk{T(t)}_{t\ge 0}$ be the translation semigroup on $L^1(\re,X^*)$ and $L^1(\re,X)$. Then,
$$
\Funk{i_1}{BUC(\re,X)}{L^1(\re,X^*)^{\sun}}{f}{\bk{g\mapsto \int_{\re}<g(\tau),f(\tau)>d\tau}}
$$
and
$$
\Funk{i_2}{BUC(\re,X^*)}{L^1(\re,X)^{\sun}}{f}{\bk{g\mapsto \int_{\re}<g(\tau),f(\tau)>d\tau}}
$$

are isometries ($\nri{f}=\sup_{g\in B_{L^1(\re,X^*)}}
\btr{\int_{\re}<g(\tau),f(\tau)>d\tau})$. Consequently, every uniformly closed subspace of $BUC(\re,X)$ is a closed subspace of $(L^1(\re,X^*)^*,\nrm{\cdot})$ and $ BUC(\re,X^*)\subset L^1(\re,X^*)^*$.
\end{pro}
\begin{proof}
Let $\seq{t}{m}\subset \re,$ such that $\ilm{m}\nrm{f(t_m)}=\nri{f}.$ By the Hahn-Banach
theorem, we find that $\seq{x^*}{m}\subset B_X^*$ with $x_m^*(f(t_m))=\nrm{f(t_m)}$
(or choose $\ep_1$ such that $x_m(f(t_m))=\nrm{f(t_m)} -\ep_1$). Let $\vp\in
C^{\infty}_0(\re)$ with $\int_{\re}\vp(t)dt=1.$ Then define $\vp_{\ep}(t)=\frac{1}{\ep}
\vp(\frac{t}{\ep}).$ If $g_{\ep,m}(s):=\vp_{\ep}(s-t_m)x_m^*,$ then all $g_{m,\ep} \in L^1(\re,X^*)$ have
norm one, and the claim is proved using regularization arguments for the bounded and uniformly continuous functions having the same modulus of continuity ($\bk{t\mapsto x_m^*f(t)}$).
%Editor2: On the line below, please ensure that the intended meaning has been maintained in this edit.
Hence,    $\int_{\re} g_{m,\ep}fd\mu\to \nrm{f(t_m)},$ and $\int_{\re}g_{m,\ep}fd\mu\to \nrm{f(t_m)-\ep_1}.$
The second claim uses the same methods as those we used to find $\seq{x}{m}\subset B_X$ such that $\nrm{f(t_m)}-1/m\ge f(t_m)(x_m).$
\end{proof}

\begin{remk}
Considering the translation semigroup and using the topology on $BUC(\re,X)$ that comes with $L^1(\re,X^*),$ in view of Prop. \ref{BUC_closed_in_Dual},  it becomes straightforward to consider the embedding
$$
\Funk{\iota}{BUC(\re,X)}{BUC(\re,X^{**})}{f}{j_Xf.}
$$
This leads to a dual semigroup representation on $BUC(\re,X^{**})$ that brings a restriction of the sun-dual of the translation semigroup defined on $L^1(\re,X^*).$
\end{remk}

With the above view we obtain for general Banach spaces and
$$BUC_{wrc}(\re,X):=\bk{f\in BUC(\re,X):f(\re) \mbox{ weak relatively compact }}$$
\begin{cor} 
The translation semigroup on $BUC_{wrc}(\re,X)$ is a dual representation, the $weak-compact-open-$topology ($x^*f_{\gamma}$ compact open convergent) is representation compatible, and there are splittings 
$$BUC_{wrc}(\re,X)=BUC_{wrc}(\re,X)_a\oplus BUC_{wrc}(\re,X)_0.$$
\end{cor}

The next example shows that weak compactness is essential.

\begin{exa} 
In general $\overline{\bk{f_t:t\in \rep}}^{\sigma(BUC(\rep,X),L^1(\rep,X^*))}\subset BUC(\rep,X) $ is not given. Therefore let $X^*$ separable,
$x^{**}\in X^{**}\backslash X$ and $\seq{x}{n}\subset X$ with $w^*-\ilm{n}x_n=x^{**}.$ With this setting we define
$$
\begin{array}{l|l}
\Funk{\gamma}{\rep}{\rep}{t}
{ \left\{\begin{array}{rcl} 
-\btr{4(t-\frac{1}{4})}+1     &:& t\in [0,\frac{1}{2}] \\
0     &:& \mbox{otherwise},
\end{array}\right .} 
&
\Funk{f}{\rep}{X}{t}{\sum_{l=1}^{\infty}\gamma(t-l+\frac{1}{4})x_l }
\end{array}
$$
Then $w^*-\ilm{t}f_t =g_{\tau},$ for all $s\in\rep,$ and $\tau\in[0,1),$ with
$$
\Funk{g}{\rep}{X^{**}}{t}{x^{**}\sum_{l=1}^{\infty}\gamma(t-l+\frac{1}{4})}
$$
which is periodic with the period $1,$ and therefore an element in $BUC(\rep,X^{**})_a,$ and fails to be an element of $BUC(\rep,X).$ As for some net $\net{t}{\gamma}{\Gamma}\subset \rep,$ we have  $Qf=\netlim{\gamma}{\Gamma}T(t_{\gamma})f=g_{\tau}$ for some $\tau\in [0,1],$ we complete the example.
\end{exa}

\begin{proof}

We have for $h\in L^1(\rep,X),$
\begin{eqnarray*}
\btr{\int_{\re}<(f_n-g)(s),h(s)>ds} &\le& \int_{n}^{n+\frac{1}{2}}\btr{<x_{n}-x^{**},h(s)>}ds \\
&\le& \int_{\rep}\btr{<x_{n}-x^{**},h(s)>} ds \to 0,
\end{eqnarray*}
by Lebesgue Domintaed Convergence theorem.
As $X^*$ is assumed to be separable, $X$ is, and therefore $L^1(\rep,X).$ Consequently, the $w^*-$topology coming with $L^1(\re,X)$ is metrizable on bounded sets. 
For a given sequence in $\rep,$  $t_{l}=n_{l}+\tau_{l},$ with $\tau_{l}\in [0,1].$ Note that, from the method of Proposition (\ref{BUC_closed_in_Dual}), $\bk{T(t)}_{t\ge0}$ is dual to a strongly continuous semigroup $\bk{T_*(t)}_{t\ge 0}$ on $L^1(\rep,X).$ Assuming $\tau_{l}\to \tau,$  we have
\begin{eqnarray*}
<T(t_{l})f,h>&=&<T(n_{l}+\tau_{l})f,h>=<T(n_l)f,T_*(\tau_l)h> \\
&&\to <g,T_*(\tau)h>=<T(\tau)g,h>.
\end{eqnarray*}
Hence $Qf=g_{\tau}$ for some $\tau\in[0,1),$ which fails to be an element of $BUC(\rep,X).$
\end{proof}

\begin{remk}
The missing weak compactness may get even worse, defining the function $f$ in the above setting with $\bk{x_k}_{k=1}^n\subset f([2^n,2^{n+1})),$ and $\seq{x}{n}$ dense in $B_X.$
\end{remk}

\section{Separability of Orbits}

Next, we present the main lemma providing the norm separability of the $w^*$-closure of the orbit.

\begin{lem} \label{Norm-separability-of-w-star-closure}
Let $f\in BUC(\re,X^{**}),$ $\semig{T}{\re}$ be the translation group, and $\overline{f(\re)}^{w^*}$ be norm-separable in $X^{**}$; for some $s\in \re,$ let
$$
\Funk{F}{\fk{(B_{X^*},w^*)\times (\scST,w^*)}}{\ce}{(x^*,S)}{<x^*,(Sf)(s)>}
$$
be continuous. Then, $\overline{O(f)}^{w^*}$ is separable in $BUC(\re,X^{**}).$
\end{lem}

\begin{proof}
By Lemma \ref{BUC-X-star-is-dual-representation}, the semigroup of translations on $BUC(\re,X^{**})$ is a dual representation; therefore, let $\scST$ be the $w^*-$closure of the translations in $L(BUC(\re,X)).$ For $f\in BUC(\re,X^{**})$,
$$
\Funk{j}{\overline{O(f)}^{w^*}}{(C((B_{X^*},w^*)\times \scST))}{f}{\bk{(x^*,S)\mapsto <x^*,(Sf)(s)>}}
$$
fulfills $\nri{jf}=\nri{f}.$
By assumption, let $\overline{\seq{x}{n}}^{\nrm{\cdot}}= \overline{f(\re)}^{w^*},$ $x^*\in B_{X^*}, $ and $S\in \scST.$ Then for $\ep >0$,
$\btr{<x^*,(Sf)(s)-x_n>}\le \nrm{x^*}\nrm{(Sf)(s)-x_n}\le \ep$ for a suitable $x_n.$
Hence, $j(\overline{O(f)}^{\tau_{wco}})$ is pointwise separable, and application of Lemma \ref{separability-lemma} gives
$$
\overline{O(f)}^{w^*} \subset C((B_{X^*},w^*)\times \scST),\nri{\cdot}) \mbox{ is norm separable.}
$$
That is, there is $\seq{f}{n}\subset \overline{O(f)}^{w^*}$ such that for every $g \in \overline{O(f)}^{w^*},$
$$
\sup_{x^*\in B_{X^*},S\in\scST}<x^*,(Sf_i)(s)-(Sg)(s)>=\sup_{S\in\scSTOS}\nrm{(Sf_i)(s)-(Sg)(s)}\le \epsilon
$$
for some $i\in\za.$
Hence,
$$
\nri{f_i-g}=\sup_{t\in\re}\nrm{T(t)f_i(s)-T(t)g(s)}\le \sup_{S\in\scSTOS}\nrm{(Sf_i)(s)-(Sg)(s)}\le \epsilon,
$$
and the proof is complete.

\end{proof}

\begin{pro} \label{WW-is-a-dual-representation}
Let $\jz\in\bk{\re,\rep},$ and let there be a closed subspace $M$ with $X\subset M\subset X^{**}$.
\begin{eqnarray*}
WAP_M(\re,X^*)&:=&\bk{f\in BUC(\re,X^*): \bk{t\to <x^{**},f(t)>}\in AP(\re) \mbox{ for all } x^{**}\in M},\\
WW_M(\jz,X^*)&:=&\bk{f\in BUC(\jz,X^*): \bk{t\to <x^{**},f(t)>}\in W(\jz) \mbox{ for all } x^{**} \in M}.
\end{eqnarray*}
Define by the seminorms the topology
$$
\tau_{M-\infty}:=\bk{p_{x_1^{**},..x_n^{**}}(f):=\sup_{t\in\re, 1\le i\le n}\btr{<x^*_i,f(t)>}: \bk{x_i^{**}}_{i=1}^n\subset M}.
$$
Then, the translation semigroup on $WW_M(\jz,X^*)$ is a dual semigroup representation with a representation-compatible topology $\tau_{M-\infty},$ and  $WW(\jz,X^*)_a=WAP(\re,X^*)_{|\jz}.$
\end{pro}
\begin{proof}
Let $\bk{t^i_{\gamma}}_{\gamma\in\Gamma,i\in\za}\subset \re,$  $\bk{\la^i_{\gamma}}_{\gamma\in\Gamma,i\in\za}\subset \ce,$ and $\net{n}{\gamma}{\Gamma}\subset \za,$ with $\sum_{i=1}^{n_{\gamma}}\btr{\la^i_{\gamma}}=1.$

Using the weak topology on $W(\jz)$ and the fact that for weakly compact sets, the absolute convex hull is weakly compact, an application of Tychonov's theorem gives
$$
\bk{\bk{\sum_{1}^{n_{\gamma}}\la^i_{\gamma}f_{t+t_i^{\gamma}}}_{t\in\re}}_{\gamma\in\Gamma}\subset \Pi_{x\in M} \
\fk{\overline{ac}^{w}\bk{<x,f_t>:t\in\re},w},
$$
which leads to $\tau_{M-\infty}$ being a convergent subnet. Applying Lemma \ref{pointwise-weak-star-stronger-than-L1}, the compatibility of the topologies on bounded sets
$$\sigma(WW(\jz,X^*),L^1(\re,X))\subset \tau_{w^*-co}\subset \tau_{X-\infty}\subset \tau_{M-\infty}$$
concludes the proof.
\end{proof}

\begin{pro} Let $X$ be a Banach space and $\semig{T}{\re}$ be the translation group. Then,
$$WAP_X(\re,X^*)\subset BUC(\re,X^*)_a.$$
\end{pro}
\begin{proof}
From the vector-valued splitting, we learn that $g(t)=<x,f(t)>=<x,f_a(t)>+<x,f_0(t)>,$ with $<x,f_0>$ a flight vector and $<x,f_a>$ reversible with respect to the scalar translation semigroup. Using the fact that almost-periodic functions are reversible, we find that $<x,f>-<x,f_a>=<x,f_0>$ for all $x\in X,$; therefore, $f=f_a,$ by Prop. \ref{rev-equals-flight-equals-zero}.
\end{proof}

\begin{theo} \label{main-separable-AP-Thm}
  Let $f\in WW_{X^*}(\re,X^{**})_a$ and $\overline{f(\re)}^{w^*}$ be norm-separable,
  %Editor2: On the line below, please ensure that the intended meaning has been maintained in this edit.
and for some $s\in \re,$ let
$$
\Funk{F}{\fk{(B_{X^*},w^*)\times (\scST,w^*)}}{\ce}{(x^*,S)}{<x^*,(Sf)(s)>}
$$
be continuous. Then, $f\in AP(\re,X^{**}).$
\end{theo}
\begin{proof}
By Prop. \ref{WW-is-a-dual-representation},	we are in the situation of Thm. \ref{separable-X_uds}. Lemma \ref{Norm-separability-of-w-star-closure} verifies the separability condition.
\end{proof}

The next corollary seems to be in contradiction with the example given in \cite[4. p. 75]{LevitanZhikov}, but the constructed counterexample fails to be uniformly continuous. It seems that they construct on an interval $I$ a sequence of scalar-valued norm-one functions with disjoint supports and extend them to all of the reals. Then, these functions have a common period, which is exactly the length of the interval. Then, they define
$$
\Funk{f}{\re}{l^2(\za)}{t}{\bk{\phi_k(t)}_{k\in\za}.}
$$
The function is weakly almost periodic in their sense, applying \cite[XII, p.51]{AmerioProuse}, but it is not an element of $WW_{l^2(\za)}(\re,l^2(\za)).$
Let $\phi_k(t_k)=1.$ Choose $s_k$ in the boundary of $\bk{t:\phi_k(t)\not= 0},$ with
$ [s_k,t_k]\subset  \supp{\phi_k}.$ By the boundedness of $I,$ we have $t_k-s_k \to 0,$ but
$$
\nrm{f(t_k)-f(s_k)}_2=\btr{\phi_k(t_k)-\phi_k(s_k)}=\btr{1-0}=1.
$$
Hence, $f$ fails to be uniformly continuous, and therefore, as they claim, it fails to be almost periodic.
\begin{cor}
Let $f\in WW_{X^*}(\re,X)_a$, let ${f(\re)}$ be weak relatively compact,
and for some $s\in \re,$ let
$$
\Funk{F}{\fk{(B_{X^*},w^*)\times (\scST,w^*)}}{\ce}{(x^*,S)}{<x^*,(Sf)(s)>}
$$
be continuous. Then, $f\in AP(\re,X).$
\end{cor}

\end{document}